\documentclass[psreqno,reqno,11pt]{amsart}
\usepackage{eurosym}
\usepackage{amssymb}
\usepackage{tikz}
\usepackage{amsmath}
\usepackage{tikz}
\usepackage{pgflibraryplotmarks}
\usepackage{wasysym}
\usepackage{stmaryrd}
\usepackage[english]{babel}

\usepackage{hyperref}

\theoremstyle{plain}
\newtheorem{theorem}{Theorem}[section]
\newtheorem{proposition}[theorem]{Proposition}

\newtheorem{corollary}[theorem]{Corollary}

\theoremstyle{definition}
\newtheorem{definition}[theorem]{Definition}

\newcommand{\la}{ T(}
\newcommand{\ra}{)}

\begin{document}
\def\sect#1{\section*{\leftline{\large\bf #1}}}
\def\th#1{\noindent{\bf #1}\bgroup\it}
\def\endth{\egroup\par}

\title[Vector Semi-Inner Products]{
Vector Semi-Inner Products}
\author{K. Rose}
\address{Department of Mathematics, Lyon College, Batesville,
AR 72501, USA}
\email{kjrose017@gmail.com}
\author{C. Schwanke}
\address{Department of Mathematics and Applied Mathematics, University of Pretoria, Private Bag X20, Hatfield 0028, South Africa and Unit for BMI, North-West University, Private Bag X6001, Potchefstroom, 2520, South Africa}
\email{cmschwanke26@gmail.com}
\author{Z. Ward}
\address{Department of Mathematics, Lyon College, Batesville,
	AR 72501, USA}
\email{zdward904@gmail.com}
\date{\today}
\subjclass[2020]{46A40}
\keywords{vector lattice, semi-inner product, Pythagorean theorem, parallelogram law}

\begin{abstract}
We formalize the notion of vector semi-inner products and introduce a class of vector seminorms which are built from these maps. The classical Pythagorean theorem and parallelogram law are then generalized to vector seminorms that have a geometric mean closed vector lattice for codomain. In the special case that this codomain is a square root closed, semiprime $f$-algebra, we provide a sharpening of the triangle inequality as well as a condition for equality.
\end{abstract}

\maketitle

\section{Introduction}\label{S:intro}

Though not formalized previously, semi-inner products with a vector lattice codomain have been proven to satisfy the Cauchy-Schwarz inequality in certain settings. Specifically, it was proven in \cite[Corollary 4]{BusvR4} that the Cauchy-Schwarz inequality holds for semi-inner products with an almost $f$-algebra codomain. This result inspired the paper \cite{BusSch3}, whose main theorem illustrates that the Cauchy-Schwarz inequality also holds for semi-inner products whose codomain is a geometric mean closed vector lattice.

These results promised the development of a vector lattice-valued semi-inner product space theory, which is the focus of this paper. Using these generalized semi-inner products, which we call \textit{vector semi-inner products}, we construct vector seminorms, as studied in \cite{BusDev}. We then elucidate how a large class of these vector seminorms satisfy an inequality that is sharper than the triangle inequality. From this result, a new equality condition for the triangle inequality is established. Finally, we tap into the theory of square mean closedness in vector lattices, which was developed in \cite{Az,AzBoBus,BusSch,dS}, to prove that a wide range of vector seminorms built from vector semi-inner products satisfy the Pythagoran theorem and the parallelogram law. We proceed with the preliminaries.

\section{Preliminaries}\label{S:prelims}

We refer the reader to \cite{AB,LuxZan1,Zan2} for basic terminology and theory for vector lattices (also called Riesz spaces) and $f$-algebras. Briefly, a vector lattice is an ordered vector space over $\mathbb{R}$ which is closed under finite suprema and infima. Given a vector lattice $F$, we as usual denote the positive cone by
\[
F^+:=\{x\in F\ :\ x\geq 0\}.
\]
All vector lattices in this paper are assumed to be Archimedean, meaning that
\[
\inf\{n^{-1}u\ :\ n\in\mathbb{N}\}=0
\]
holds for all elements $u$ of the positive cone. An $f$-algebra $F$ is a vector lattice equipped with a ring multiplication (which we denote by juxtaposition as usual) for which
\begin{itemize}
    \item[(i)] $xy\in F^+$ for every $x,y\in F^+$, and 
    \item[(ii)] $\inf\{x,y\}=0$ implies $\inf\{(xz),y\}=\inf\{x,(yz)\}=0$ for all $z\in F^+$.
\end{itemize}
We add that the multiplication on any Archimedean $f$-algebra is commutative \cite[Theorem~140.10]{Zan2}.

A vector lattice $F$ is said to be \textit{geometric mean closed} (see \cite[page 486]{AzBoBus}) if
\[
\inf\{\theta u+\theta^{-1}v:\theta\in(0,\infty)\}
\]
exists in $F$ for every $u,v\in F^{+}$. Given a geometric mean closed vector lattice $F$ and $u,v\in F^+$ we define
\[
u\boxtimes v:=2^{-1}\inf\{\theta u+\theta^{-1}v:\theta\in(0,\infty)\}.
\]

It will benefit the reader to note that for $u,v\in\mathbb{R}^+$, the expression $u\boxtimes v$ coincides with the classical geometric mean:
\[
u\boxtimes v=\sqrt{uv}\quad (u,v\in\mathbb{R}^+).
\]

Every geometric mean closed vector lattice is also square mean closed \cite[Theorem~4.4]{AzBoBus}. We say a vector lattice $F$ is \textit{square mean closed} (see \cite[page 482]{AzBoBus}) if
\[
\sup\{ (\cos\theta)u+(\sin\theta)v:\theta\in[0,2\pi]\}
\]
exists in $F$ for every $u,v\in F$. In this case we define
\[
u\boxplus v=\sup\{ (\cos\theta)u+(\sin\theta)v:\theta\in[0,2\pi]\}\quad (u,v\in F).
\]

To again aid the reader, we note that
\[
u\boxplus v=\sqrt{u^2+v^2}
\]
holds for all $u,v\in\mathbb{R}$.

In the special case that $u$ and $v$ are positive, we have (see \cite[Theorem 91.4(ii)]{Zan2})
\[
u\boxplus v=\sup\{ (\cos\theta)u+(\sin\theta)v:\theta\in[0,2^{-1}\pi]\}\quad (u,v\in F^+).
\]

Given an $f$-algebra $F$ and $a\in F$, we again as customary write $a^2=aa$. For $a\in F^+$, if there exists a unique $b\in F^{+}$ for which $b^{2}=a$, we write $b=\sqrt{a}$. If $F$ is a \textit{semiprime} $f$-algebra (meaning free of nilpotents) and $a,b\in F^{+}$ satisfy $b^{2}=a$ then we have $b=\sqrt{a}$ by \cite[Proposition 2(ii)]{BeuHui}.

Every vector space in this document is assumed to be over $\mathbb{R}$.

\begin{definition}\label{D:SIP}
        Let $V$ be a vector space, and let $F$ be an ordered vector space. We call a map $T\colon V\times V \to F$ a \textit{vector semi-inner product} if 
        \begin{itemize}
            \item [(i)] $T(x+y, z)=T(x,z)+T(y,z)\quad (x, y, z \in V),$
            \item [(ii)] $T(x,y + z) = T(x, y) + T(x, z)\quad (x, y, z \in V),$
            \item [(iii)] $\lambda T(x, y) = T(\lambda x, y) = T(x, \lambda y)\quad (x, y \in V,\ \lambda \in \mathbb{R}),$
            \item [(iv)] $T(x, y) = T(y, x)\quad (x, y \in V), $ and 
            \item [(v)] $T(x, x) \ge 0 \quad (x\in V).$
        \end{itemize}
    \end{definition}

The following theorem regarding vector semi-inner products is the content of \cite[Theorem 3.1]{BusSch3} and is an essential ingredient to the results in Section 3.

\begin{theorem}\label{T:CSI} \textnormal{\textbf{(Cauchy-Schwarz Inequality)}}
Let $V$ be a vector space, and suppose that $F$ is a geometric mean closed vector lattice. If $T\colon V\times V\rightarrow F$ is a vector semi-inner product, then
\begin{itemize}
	\item[(1)] $\underset{\lambda\in\mathbb{R}\setminus\{ 0\}}{\inf}\{|\lambda|^{-1}T(\lambda x-y,\lambda x-y)\}$ exists in $F\quad (x,y\in V)$,
	\item[(2)] $|T(x,y)|=T(x,x)\boxtimes T(y,y)-2^{-1}\underset{\lambda\in\mathbb{R}\setminus\{0\}}{\inf}\{|\lambda|^{-1}T(\lambda x-y,\lambda x-y)\}\quad (x,y\in V)$,
	\item[(3)] $|T(x,y)|\leq T(x,x)\boxtimes T(y,y)\quad (x,y\in V)$, and
	\item[(4)] $|T(x,y)|=T(x,x)\boxtimes T(y,y)$ if and only if $\underset{\lambda\in\mathbb{R}\setminus\{ 0\}}{\inf}\{|\lambda|^{-1}T(\lambda x-y,\lambda x-y)\}=0$.
\end{itemize}
\end{theorem}

In Section 3 we also utilize the following proposition by Azouzi, Boulabiar, and Buskes.

\begin{proposition}\cite[Lemma~5.1]{AzBoBus}\label{P:AzBoBus}
Let $F$ be a geometric mean closed vector lattice, put $\lambda\in\mathbb{R}^+$, and let $a,b,c\subseteq F^+$. The following hold:
\begin{itemize}
	\item[(1)] (Biadditvity) $(a+b)\boxtimes c=(a\boxtimes c)\boxplus(b\boxtimes c)$\quad and\quad $a\boxtimes(b+c)=(a\boxtimes b)\boxplus(a\boxtimes c)$;	\item[(2)] (Separate Positive Homogeneity) $(\lambda a)\boxtimes b =\lambda^{1/2}(a\boxtimes b)=a\boxtimes(\lambda b)$. 
\end{itemize}
\end{proposition}

\section{Vector Semi-Norms Via Vector Semi-Inner Products}\label{S:TTI}
    We construct vector semi-norms from vector semi-inner products in this section and present our results. To begin, we consider the following definitions. 
    
    \begin{definition}\label{D:VSN}
    Given a vector space $V$ and an ordered vector space $F$, a map $\|\cdot\|\colon V\to F$ is called a \textit{vector seminorm} (see \cite[Section 2]{BusDev}) if
        
        (Positivity)\ $\|x\|\in F^+\quad (x\in V)$,
        
        (Absolute Homogeneity)\ $\| \alpha x\| = |\alpha|\|x\|\quad (x\in V, \alpha\in\mathbb{R})$, and
        
        (Triangle Inequality)\ $\|x+y\|\leq\|x\|+\|y\|\quad (x,y\in V)$.
    \end{definition}
    
    \begin{definition}\label{D:Norm}
        Let $V$ be a vector space, $F$ be a geometric mean closed vector lattice, and suppose that $T\colon V\times V \to F$ is a vector semi-inner product. Put $u \in F^+$. Define 
        \[
            \|x\|^T_u := T(x,x) \boxtimes u \quad (x \in V).
        \]
    \end{definition}
    
    We prove next that the maps defined in Definition~\ref{D:Norm} above are in fact vector seminorms.
    
    \begin{theorem}\label{T:TisaVSN}
    If $V$ is a vector space, $F$ is a geometric mean closed vector lattice, $u \in F^+$, and $T\colon V\times V \to F$ is a vector semi-inner product, then $\|\cdot\|^T_u$ is a vector seminorm.
    \end{theorem}
    
    \begin{proof} Let $V$ be a vector space and $F$ be a geometric mean closed vector lattice. Suppose that $T\colon V\times V \to F$ is a vector semi-inner product and that $u \in F^+$. The positivity of $\|\cdot\|$ is evident. We first prove the absolute homogeneity of $\|\cdot\|^T_u$. To this end, put $\alpha \in \mathbb{R}$ and $x\in V$. The desired result $||\alpha x||_u^{T}=|\alpha|||x||_u^{T}$ is trivial in the case that $\alpha=0$. Suppose $\alpha \ne 0$. We want to show that 
         \[
        \| \alpha x\|^T_u = |\alpha|\|x\|^T_u,
        \]
         or equivalently, that
         \[
            \underset{\theta > 0}{\inf}\{\theta\la\alpha x, \alpha x\ra + \theta^{-1}u\} = |\alpha|\underset{\theta > 0}{\inf}\{\theta\la x,x\ra  +\theta^{-1}u\}.
        \]
        For this task, observe that 
        \setlength{\abovedisplayskip}{5pt}
        \setlength{\belowdisplayskip}{15pt}
        \setlength{\abovedisplayshortskip}{0pt}
        \setlength{\belowdisplayshortskip}{0pt}
        \begin{align*}
            & \underset{\theta > 0}{\inf}\{\theta\la\alpha x, \alpha x\ra + \theta^{-1}u\}= \underset{\theta > 0}{\inf}\{\theta |\alpha|^2\la x,  x\ra + \theta^{-1}u\}.
        \end{align*}
        Then we have
        \setlength{\abovedisplayskip}{5pt}
        \setlength{\belowdisplayskip}{15pt}
        \setlength{\abovedisplayshortskip}{0pt}
        \setlength{\belowdisplayshortskip}{0pt}
        \begin{align*}
            \underset{\theta > 0}{\inf}\{\theta |\alpha|^2\la x,  x\ra + \theta^{-1}u\} = |\alpha |\underset{\theta > 0}{\inf}\{\theta | \alpha |\la x,  x\ra + (\theta |\alpha |)^{-1}u\}.
        \end{align*}
        Substituting $t:=\theta | \alpha |$, we get
        \[
            |\alpha |\underset{\theta > 0}{\inf}\{\theta | \alpha |\la x,  x\ra + (\theta |\alpha |)^{-1}u\} = | \alpha |\underset{t > 0}{\inf}\{t\la x,x\ra + t^{-1}u\},
        \]
        as desired. 
  
We proceed by showing that $\|\cdot\|^T_u$ satisfies the triangle inequality. For this purpose, let $x,y\in V$. We will show that $2\|x\|^T_u + 2\|y\|^T_u - 2\|x+y\|^T_u \geq 0$.
        To this end, observe that
        \setlength{\abovedisplayskip}{5pt}
        \setlength{\belowdisplayskip}{15pt}
        \setlength{\abovedisplayshortskip}{0pt}
        \setlength{\belowdisplayshortskip}{0pt}
        \begin{align*}
            & 2\|x\|^T_u + 2\|y\|^T_u - 2\|x+y\|^T_u \\
            & = \underset{\theta > 0}{\inf}\{\theta\la x,x\ra  + \theta^{-1}u\} + \underset{\psi > 0}{\inf}\{\psi\la y,y \ra + \psi^{-1}u\} - \underset{\lambda > 0}{\inf}\{\lambda\la x+y,x+y\ra + \lambda^{-1}u \} \\
            & = \underset{\theta, \psi > 0}{\inf}\{\theta \la x, x\ra + \theta^{-1}u + \psi \la y,y\ra + \psi^{-1}u\}\\
            &\qquad\qquad\qquad- \underset{\lambda > 0}{\inf}\{\lambda\la x,x\ra + 2\lambda\la x,y\ra + \lambda\la y,y\ra + \lambda^{-1}u\} \\
            & \geq \underset{\theta, \psi > 0}{\inf}\{\theta \la x, x\ra + \theta^{-1}u + \psi \la y,y\ra + \psi^{-1}u\}\\
            &\qquad\qquad\qquad- \underset{\lambda > 0}{\inf}\{\lambda\la x,x\ra + 2\lambda | \la x,y\ra | + \lambda\la y,y\ra + \lambda^{-1}u\}.
        \end{align*}
        Using Theorem~\ref{T:CSI}(3), we obtain 
        \begin{align*}
            & \underset{\theta, \psi > 0}{\inf}\{\theta \la x, x\ra + \theta^{-1}u + \psi \la y,y\ra + \psi^{-1}u\} - \underset{\lambda > 0}{\inf}\{\lambda\la x,x\ra + 2\lambda | \la x,y\ra | + \lambda\la y,y\ra + \lambda^{-1}u\} \\
            & \geq \underset{\theta, \psi > 0}{\inf}\{\theta \la x, x\ra + \theta^{-1}u + \psi \la y,y\ra + \psi^{-1}u\}\\
            &\qquad\qquad\qquad- \underset{\lambda > 0}{\inf}\{\lambda\la x,x\ra + 2\lambda (T(x,x) \boxtimes T(y,y)) + \lambda\la y,y\ra + \lambda^{-1}u\} \\
            & = \underset{\theta, \psi > 0}{\inf}\{\theta \la x, x\ra  + \psi \la y,y\ra + (\theta^{-1} + \psi^{-1})u\}\\
            &\qquad\qquad\qquad - \underset{\lambda, \phi > 0}{\inf}\{\lambda(1 + \phi)\la x,x\ra + \lambda(1 + \phi^{-1})\la y,y\ra + \lambda^{-1}u\}.
        \end{align*}
        To show that the above difference is an element of $F^{+}$, it suffices to show the inclusion
        \setlength{\abovedisplayskip}{5pt}
        \setlength{\belowdisplayskip}{15pt}
        \setlength{\abovedisplayshortskip}{0pt}
        \setlength{\belowdisplayshortskip}{0pt}
        \begin{align*}
            &\{\theta\la x,x\ra + \psi\la y,y\ra + (\theta^{-1} + \psi^{-1})u : \theta, \psi \in (0, \infty)\} \\
            & \subseteq \{\lambda(1 + \phi)\la x,x\ra + \lambda(1 + \phi^{-1})\la y,y\ra + \lambda^{-1}u : \lambda, \phi \in (0, \infty)\}.
        \end{align*}
        and use the readily-checked monotonicity of infima: $\inf{A}\ge\inf{B}$ whenever $A\subseteq B\subseteq F$ and $\inf{A}$ and $\inf{B}$ exist. To this end, let $\theta, \psi > 0$. Set $\lambda := (\theta^{-1} + \psi^{-1})^{-1}$ and $\phi := \theta\psi^{-1}$, and note that $\lambda$, $\phi > 0$. Observe that $\lambda(1 + \phi) = \theta$ and $\lambda(1 + \phi^{-1}) = \psi$. Therefore, we obtain the inclusion above. 
    \end{proof}
    
    In the theorem above, the triangle inequality can be sharpened in the case that $F$ is a semi-prime, square root closed $f$-algebra. A semiprime $f$-algebra $F$ is called \textit{square root closed} if for every $a\in F^+$ there exists a (unique) $b\in F^+$ for which $b=\sqrt{a}$.
    
\begin{theorem}[\textbf{Sharpened Triangle Inequality}]\label{T:STI} Let $V$ be a vector space, and suppose $F$ is a square root closed, semi-prime $f$-algebra. Let $T: V \times V \rightarrow F$  be a vector semi-inner product, and put $u \in F^+$. If $x, y \in V$, then 
\begin{align*}
||x + y||^T_u\leq\sqrt{(||x||^T_u+||y||^T_u)^2-\underset{\lambda\in\mathbb{R}\setminus\{0\}}{\inf}\{|\lambda|^{-1}(||\lambda x-y||^T_u)^2\}}\leq||x||^T_u+||y||^T_u.
\end{align*}
Moreover, the equality
\[
||x + y||^T_u=\sqrt{(||x||^T_u+||y||^T_u)^2-\underset{\lambda\in\mathbb{R}\setminus\{0\}}{\inf}\{|\lambda|^{-1}(||\lambda x-y||^T_u)^2\}}
\]
holds if and only if $T(x,y)u\in F^+$.
\end{theorem}

\begin{proof}
Let $x,y\in V$. By \cite[Proposition 2.1]{BusSch3}, which is a slight generalization of \cite[Theorem 2.21]{Az},
\[
a\boxtimes b=\sqrt{ab}\quad (a,b\in F^+)
\]
holds, where the calculation $a\boxtimes b$ is taken in the geometric mean completion of $F$ (see \cite{BusSch}), where it is guaranteed to exist. But since $F$ is square root closed, we conclude that $a\boxtimes b\in F$, i.e. $F$ is itself geometric mean closed. It then follows from \cite[Theorem 142.3(ii)]{Zan2} that
\[
(\|z\|^T_u)^2=(\sqrt{T(z,z)u})^2=T(z,z)u
\]
holds for all $z\in V$. Using this fact, the bilinearity of $T$, Theorem~\ref{T:CSI}(2), and the order continuity and commutativity of Archimedean $f$-algebra multiplication, we obtain
\begin{align*}
&(\|x+y\|^T_u)^2=(\|x\|^T_u)^2+(\|y\|^T_u)^2+2T(x,y)u\\
&\leq(\|x\|^T_u)^2+(\|y\|^T_u)^2+2|T(x,y)|u\\
&=(\|x\|^T_u)^2+(\|y\|^T_u)^2+2\left(\sqrt{T(x,x)T(y,y)}-2^{-1}\underset{\lambda\in\mathbb{R}\setminus\{0\}}{\inf}\{|\lambda|^{-1}T(\lambda x-y,\lambda x-y)\}\right)u\\
&=(\|x\|^T_u)^2+(\|y\|^T_u)^2+2\|x\|^T_u\|y\|^T_u-\underset{\lambda\in\mathbb{R}\setminus\{0\}}{\inf}\{|\lambda|^{-1}(\|\lambda x-y\|^T_u)^2\}\\
&=(\|x\|^T_u+\|y\|^T_u)^2-\underset{\lambda\in\mathbb{R}\setminus\{0\}}{\inf}\{|\lambda|^{-1}(\|\lambda x-y\|^T_u)^2\}.
\end{align*}
Invoking \cite[Theorem 142.3(ii)]{Zan2} once more, we obtain
\[
\|x+y\|^T_u\leq\sqrt{(\|x\|^T_u+\|y\|^T_u)^2-\underset{\lambda\in\mathbb{R}\setminus\{0\}}{\inf}\{|\lambda|^{-1}(\|\lambda x-y\|^T_u)^2\}}.
\]
That the equality condition holds if and only if $T(x,y)u\in F^+$ is evident from the string of relations above. Furthermore, the inequality
\[
\sqrt{(\|x\|^T_u+\|y\|^T_u)^2-\underset{\lambda\in\mathbb{R}\setminus\{0\}}{\inf}\{|\lambda|^{-1}(\|\lambda x-y\|^T_u)^2\}}\leq\|x\|^T_u+\|y\|^T_u
\]
also follows from \cite[Theorem 142.3(ii)]{Zan2}.
\end{proof}

If $F$ is instead a geometric mean closed semi-prime $f$-algebra (and not necessarily square root closed), then the proof of Theorem~\ref{T:STI} also verifies the following reformulation of our sharpened triangle inequality.

\begin{theorem}[\textbf{Sharpened Triangle Inequality}]\label{T:STIv2}
Let $V$ be a vector space, and suppose $F$ is a geometric mean closed, semi-prime $f$-algebra. Let $T: V \times V \rightarrow F$  be a vector semi-inner product, and put $u \in F^+$. If $x, y \in V$, then 
\[
(||x + y||^T_u)^2\leq(||x||^T_u+||y||^T_u)^2-\underset{\lambda\in\mathbb{R}\setminus\{0\}}{\inf}\{|\lambda|^{-1}(||\lambda x-y||^T_u)^2\}\leq(||x||^T_u+||y||^T_u)^2.
\]
Moreover, the equality
\[
(||x + y||^T_u)^2=(||x||^T_u+||y||^T_u)^2-\underset{\lambda\in\mathbb{R}\setminus\{0\}}{\inf}\{|\lambda|^{-1}(||\lambda x-y||^T_u)^2\}
\]
holds if and only if $T(x,y)u\in F^+$.
\end{theorem}

As an immediate corollary, we obtain the following characterization for when the norm $\|\cdot\|^T_u$ is additive: $||x+y||^{T}_{u} = ||x||^{T}_{u} + ||y||^{T}_{u}$.

\begin{corollary} Let $V$ be a vector space, and suppose $F$ is a geometric mean closed semi-prime $f$-algebra. Let $T: V \times V \rightarrow F$  be a vector semi-inner product, and put $u \in F^+$. If $x, y \in V$, then 
\begin{align*}
||x + y||^T_u=||x||^T_u+||y||^T_u
\end{align*}
holds if and only if $T(x,y)u\in F^+$ and $\underset{\lambda\in\mathbb{R}\setminus\{0\}}{\inf}\{|\lambda|^{-1}(\|\lambda x-y\|^T_u)^2\}=0$.
\end{corollary}

Next we utilize the square mean operation to extend the classical Pythagorean theorem for inner product spaces to a result that holds for vector seminorms that are built from vector semi-inner products with geometric mean closed codomain. Indeed, note that the conclusion of the classical Pythagorean theorem (see e.g. \cite[Theorem~2.2]{Con})
\[
\|x+y\|^2=\|x\|^2+\|y\|^2
\]
is equivalent to
\[
\|x+y\|=\|x\|\boxplus\|y\|.
\]

\begin{theorem}\label{T:TTI} \textnormal{\textbf{(The Pythagorean Theorem)}}
        Let $V$ be a vector space, and let $F$ be a geometric mean closed vector lattice. Let $T: V \times V \rightarrow F$  be a vector semi-inner product. Let $u \in F^+$. If $x, y \in V$ are such that $T(x,y) = 0$, then 
            \[
                ||x + y||^T_u = ||x||^T_u \boxplus ||y||^T_u.
            \]
    \end{theorem}
    
    \begin{proof} Let $x, y \in V$ satisfy $T(x,y)=0$. Then using Proposition~\ref{P:AzBoBus}~(1) in the fourth equality below yields
    	\begin{align*}
    		 ||x + y||^T_u&=T(x+y,x+y)\boxtimes u\\
    		 &=\Bigl(T(x,x)+2T(x,y)+T(y,y)\Bigr)\boxtimes u\\
    		 &=\Bigl(T(x,x)+T(y,y)\Bigr)\boxtimes u\\
    		 &=\Bigl(T(x,x)\boxtimes u\Bigr)\boxplus\Bigl(T(y,y)\boxtimes u\Bigr)\\
    		 &=||x||^T_u \boxplus ||y||^T_u.
    	\end{align*}
    \end{proof}

In a manner similar to the Pythagorean theorem, we use the square mean operation on vector lattices to generalize the classical parallelogram law. For this task, first observe that the conclusion of the classical parallelogram law (see e.g. \cite[Theorem~2.3]{Con})
\[
\|x+y\|^2+\|x-y\|^2=2(\|x\|^2+\|y\|^2)
\]
is equivalent to
\[
||x + y||\boxplus||x - y||=2^{1/2}(||x||\boxplus||y||).
\]

\begin{theorem}\label{T:PL} \textnormal{\textbf{(The Parallelogram Law)}}
        Let V be a vector space, and let F be a geometric mean closed vector lattice.
    	Suppose $T$ is a vector semi-inner product, and let $x, y \in V$. Let $u \in F^{+}$. Then the following holds:
    	\[
        	||x + y||^T_u \boxplus ||x - y||^T_u = 2^{1/2}(||x||^T_u \boxplus ||y||^T_u).
        \]
    \end{theorem}

\begin{proof}
Using Proposition~\ref{P:AzBoBus}~(1) in the third and sixth equalities below and Proposition~\ref{P:AzBoBus}~(2) in the fifth equality below, we obtain
\begin{align*}
    	&||x + y||^T_u \boxplus ||x - y||^T_u=\Bigl(T(x+y,x+y)\boxtimes u\Bigr)\boxplus\Bigl(T(x-y,x-y)\boxtimes u\Bigr)\\
    	&=\biggl(\Bigl(T(x,x)+2T(x,y)+T(y,y)\Bigr)\boxtimes u\biggr)\boxplus\biggl(\Bigl(T(x,x)-2T(x,y)+T(y,y)\Bigr)\boxtimes u\biggr)\\
    	&=\Bigl(T(x,x)+2T(x,y)+T(y,y)+T(x,x)-2T(x,y)+T(y,y)\Bigr)\boxtimes u\\
    	&=\Bigl(2T(x,x)+2T(y,y)\Bigr)\boxtimes u\\
    	&=2^{1/2}\biggl(\Bigl(T(x,x)+T(y,y)\Bigr)\boxtimes u\biggr)\\
    	&=2^{1/2}\biggl(\Bigl(T(x,x)\boxtimes u\Bigr)\boxplus\Bigl(T(y,y)\boxtimes u\Bigr)\biggr)\\
    	&=2^{1/2}(||x||^T_u \boxplus ||y||^T_u).
    \end{align*}
    \end{proof}

\bibliography{mybib}
\bibliographystyle{amsplain}

\end{document}